\documentclass{amsart}
\usepackage[utf8]{inputenc}
\usepackage{amssymb}
\usepackage{hyperref}
\usepackage[final]{showkeys} 

\input xy
\xyoption{all}

\theoremstyle{definition}
\newtheorem{mydef}{Definition}[section]
\newtheorem{lem}[mydef]{Lemma}
\newtheorem{thm}[mydef]{Theorem}

\newtheorem{cor}[mydef]{Corollary}

\newtheorem{defin}[mydef]{Definition}

\newtheorem{remark}[mydef]{Remark}
\newtheorem{notation}[mydef]{Notation}
\newtheorem{fact}[mydef]{Fact}

\newcommand{\fct}[2]{{}^{#1}#2}

\newcommand{\ba}{\bar{a}}
\newcommand{\bb}{\bar{b}}
\newcommand{\bc}{\bar{c}}

\newcommand{\bu}{\bar{u}}

\newcommand{\ran}[1]{\text{ran}(#1)}

\newcommand{\cf}[1]{\text{cf} (#1)}
\newcommand{\seq}[1]{\langle #1 \rangle}

\newcommand{\leap}[1]{\le_{#1}}

\newcommand{\geap}[1]{\ge_{#1}}

\newcommand{\lea}{\leap{\K}}

\newcommand{\gea}{\geap{\K}}

\newcommand{\K}{\mathbf{K}}

\newbox\noforkbox \newdimen\forklinewidth
\forklinewidth=0.3pt \setbox0\hbox{$\textstyle\smile$}
\setbox1\hbox to \wd0{\hfil\vrule width \forklinewidth depth-2pt
 height 10pt \hfil}
\wd1=0 cm \setbox\noforkbox\hbox{\lower 2pt\box1\lower
2pt\box0\relax}
\def\unionstick{\mathop{\copy\noforkbox}\limits}

\def\1nf{\unionstick^{(1)}}

\def\2nf{\unionstick^{(2)}}
\def\3nf{\unionstick^{(3)}}

\newcommand{\tp}{\text{tp}}
\newcommand{\gtp}{\text{gtp}}

\newcommand{\gS}{\text{gS}}

\newcommand{\hanfe}[1]{\beth_{\left(2^{#1}\right)^+}}

\newcommand{\EM}{\operatorname{EM}}
\newcommand{\Ll}{\mathbb{L}}
\newcommand{\otp}{\operatorname{otp}}

\newcommand{\LS}{\text{LS}}

\newcommand{\BI}{\mathbf{I}}

\title[Saturation in AECs with amalgamation]{Saturation and solvability in abstract elementary classes with amalgamation}
\date{\today}

\author{Sebastien Vasey}
\email{sebv@cmu.edu}
\urladdr{http://math.cmu.edu/\textasciitilde svasey/}
\address{Department of Mathematical Sciences, Carnegie Mellon University, Pittsburgh, Pennsylvania, USA}

\begin{document}

\begin{abstract}
  \begin{thm}\label{sat-abstract}
    Let $\K$ be an abstract elementary class (AEC) with amalgamation and no maximal models. Let $\lambda > \LS (\K)$. If $\K$ is categorical in $\lambda$, then the model of cardinality $\lambda$ is Galois-saturated.
  \end{thm}

  This answers a question asked independently by Baldwin and Shelah. We deduce several corollaries: $\K$ has a unique limit model in each cardinal below $\lambda$, (when $\lambda$ is big-enough) $\K$ is weakly tame below $\lambda$, and the thresholds of several existing categoricity transfers can be improved.

  We also prove a downward transfer of solvability (a version of superstability introduced by Shelah):

  \begin{cor}\label{solv-abstract}
    Let $\K$ be an AEC with amalgamation and no maximal models. Let $\lambda > \mu > \LS (\K)$. If $\K$ is solvable in $\lambda$, then $\K$ is solvable in $\mu$.
  \end{cor}
\end{abstract}

\maketitle

\tableofcontents

\section{Introduction}

\subsection{Motivation}
Morley's categoricity theorem \cite{morley-cip} states that if a countable theory has a unique model of \emph{some} uncountable cardinality, then it has a unique model in \emph{all} uncountable cardinalities. The method of proof led to the development of stability theory, now a central area of model theory. In the mid seventies, Shelah conjectured a similar statement for classes of models of an $\Ll_{\omega_1, \omega}$-theory \cite[Open problem D.3(a)]{shelahfobook} and more generally for abstract elementary classes (AECs) \cite[Conjecture N.4.2]{shelahaecbook}\footnote{see the introduction of \cite{ap-universal-v11-toappear} for more on the history of the conjecture.}. A key step in Morley's proof was to show that the model in the categoricity cardinal is saturated. In this paper, we lift this step to the framework of AECs which satisfy the amalgamation property and have no maximal models.

In this context, saturation is defined in terms of Galois (orbital) types. Shelah \cite[II.3.10]{sh300-orig} (see also \cite[Theorem 6.7]{grossberg2002}) has justified this definition by showing that (with the hypotheses of amalgamation and no maximal models, \emph{which we make for the remainder of this section}) this notion of saturation is \emph{equivalent} to being model-homogeneous\footnote{where $M \in \K$ is \emph{model-homogeneous} if for every $M_0, N \in \K$ with $M_0 \lea M$, $M_0 \lea N$, and $\|N\| <  \|M\|$, there is a $\K$-embedding $f: N \xrightarrow[M_0]{} M$.} (in particular there can be at most one saturated model of a given cardinality). In a milestone paper, Shelah \cite{sh394} has shown that (again assuming amalgamation and no maximal models) a downward analog of Morley's categoricity theorem holds \emph{if} the starting categoricity cardinal is high-enough \emph{and a successor}. One reason\footnote{but not the only one, see the discussion after \cite[Theorem 2.4.4]{bv-survey-v4-toappear}.} for making the successor assumption was exactly to show that the model in the categoricity cardinal was saturated. Indeed, Shelah observes \cite[Claim 1.7.(b)]{sh394} that if $\K$ is categorical in $\lambda$ and $\cf{\lambda} > \LS (\K)$, then the model of cardinality $\lambda$ is $\cf{\lambda}$-saturated. In particular, if $\lambda$ is regular then the model of cardinality $\lambda$ is saturated. Shelah \cite[Question IV.7.11]{shelahaecbook} and independently Baldwin \cite[Problem D.1.(2)]{baldwinbook09} have asked whether the model of cardinality $\lambda$ is saturated even when $\lambda$ is singular. The present paper answers positively (this is Theorem \ref{sat-abstract} from the abstract, proven here as Corollary \ref{solvable-saturated-cor}).

\subsection{Earlier work}
Shelah and Villaveces (see Fact \ref{shvi}) have shown that (regardless of the cofinality of $\lambda$) categoricity in $\lambda$ implies that a certain local superstability condition (see Definition \ref{ss-def}) holds below $\lambda$. In \cite[Theorem 5.6]{ss-tame-jsl}, we showed that the local superstability condition implies stability (defined in terms of the number of Galois types) in all cardinals if the class $\K$ is $\LS (\K)$-tame (a locality property for Galois types introduced by Grossberg and VanDieren \cite{tamenessone}). Therefore if $\K$ is $\LS (\K)$-tame and categorical in $\lambda > \LS (\K)$, then $\K$ is stable in $\lambda$ and hence the model of cardinality $\lambda$ is saturated. This gives a new proof of the corresponding first-order fact. However without assuming tameness, we \emph{cannot} in general conclude stability in the categoricity cardinal $\lambda$ (there is a counterexample due to Hart and Shelah and analyzed in details by Baldwin and Kolesnikov \cite{hs-example, bk-hs}), thus different ideas are needed.

Shelah \cite[Theorem IV.7.8]{shelahaecbook} claims that the model of cardinality $\lambda$ is $\mu^+$-saturated (for $\mu \ge \LS (\K)$) if $2^{2^{\mu}} + \aleph_{\mu^{+4}} \le \lambda$. We have not fully verified Shelah's proof, which uses PCF theory as well as the theory of Ehrenfeucht-Mostowski (EM) models\footnote{Shelah even claims that it is enough to assume amalgamation and no maximal models in a small subset of cardinal below $\lambda$, but we are unable to verify Shelah's claim that ``(i) + (iii) suffices'' in clause $(e)(\beta)$ of the theorem.}. 

With VanDieren \cite[Corollary 7.4]{vv-symmetry-transfer-afml}, we showed that the model of cardinality $\lambda$ is $\mu^+$-saturated if $\lambda \ge \beth_{\left(2^{\mu^+}\right)^+}$. Assuming the generalized continuum hypothesis (GCH), $\beth_{\left(2^{\mu^+}\right)^+} = \aleph_{\mu^{+3}}$ so the bound is better than Shelah's (but if GCH fails badly then Shelah's bound is better). We conclude that if $\lambda = \beth_\lambda$ then the model of cardinality $\lambda$ is saturated (and assuming Shelah's bound, ``$\lambda = \beth_\lambda$'' can be replaced by ``$\lambda = \aleph_\lambda$ and $\lambda$ is strong limit''). We show here that these hypotheses are not needed (the simplest new case is when $\LS (\K) = \aleph_0$ and $\lambda = \aleph_\omega$).

\subsection{Description of the proof}
The proof uses the symmetry property for splitting first isolated by VanDieren \cite{vandieren-symmetry-apal}. It follows from an earlier result of VanDieren \cite{vandieren-chainsat-apal} that if symmetry holds in a successor cardinal $\mu$ then the model in the categoricity cardinal $\lambda$ is $\mu$-saturated. Further if symmetry in $\mu$ \emph{fails} then $\K$ must satisfy a variant of the order property (defined in terms of Galois types) of length $\lambda$ \cite[Theorem 5.8]{vv-symmetry-transfer-afml}. It turns out that if the length of this order property is bigger than $\gamma := \beth_{(2^{\mu})^+}$ then $\K$ is unstable below $\lambda$ and this contradicts categoricity. The aforementioned result with VanDieren used that $\lambda \ge \gamma$. The key argument of this paper (Theorem \ref{op-thm}) shows that $\K$ will have the order property of length $\gamma$ \emph{even when $\lambda < \gamma$}.

The main ingredient is a little known result of Shelah \cite[Claim 4.15]{sh394} proving from categoricity in $\lambda$ that any sequence of length $\lambda$ contains a \emph{strictly} indiscernible subsequence. Here, indiscernible is as usual defined in terms of Galois types and an indiscernible sequence is strict when (roughly) it can be extended to a longer indiscernible sequence of arbitrary size. For the convenience of the reader (and because Shelah omits several details), we give a full proof of Shelah's claim here (Fact \ref{main-fact}).

\subsection{Solvability}
We can generalize Fact \ref{main-fact} using a weakening of categoricity called \emph{solvability} (see Definition \ref{solv-def} here). Solvability was introduced by Shelah in \cite[Chapter IV]{shelahaecbook} as a possible definition of superstability in the AEC framework (it is equivalent to superstability in the first-order case \cite[Corollary 5.3]{gv-superstability-v5-toappear}). Shelah has asked \cite[Question N.4.4]{shelahaecbook} whether the solvability spectrum satisfies an analog of the categoricity conjecture. Inspired by this question, we showed with Grossberg \cite[Theorem 5.4]{gv-superstability-v5-toappear} that the solvability spectrum is either bounded or a tail \emph{provided that} the AEC is tame (and has amalgamation and no maximal models). As an application of the main result of this paper, we show here \emph{without assuming tameness} (but still using amalgamation and no maximal models) that the solvability spectrum satisfies a \emph{downward} analog of Shelah's categoricity conjecture (this is Corollary \ref{solv-abstract} from the abstract proven here as Corollary \ref{solv-downward}). Assuming tameness, we can also improve the threshold cardinal of our aforementioned work with Grossberg (Corollary \ref{upward-solv}).

\subsection{Other applications}
Other applications of our result can be obtained by taking known theorems that assumed that the model in the categoricity cardinal had some degree of saturation, and removing this saturation assumption from the hypotheses of the theorem! Several consequences are listed in Section \ref{applications-sec}. Especially notable is that uniqueness of limit models\footnote{The reader can consult \cite{gvv-mlq} for more background and motivation on limit models.} holds everywhere below the categoricity cardinal (Corollary \ref{big-cor}.(\ref{big-cor-2}). This gives a proof of the (in)famous \cite[Theorem 3.3.7]{shvi635} (where a gap was identified in VanDieren's Ph.D.\ thesis \cite{vandierenthesis}) \emph{provided that} the class has full amalgamation. The original statement assuming only density of amalgamation bases remains open but we also make progress toward it, fixing a gap of \cite{vandierennomax} isolated in \cite{nomaxerrata}. This is presented in Section \ref{uq-limit-subsec}.

\subsection{Notes}
The background required to read the core (i.e.\ the first four sections) of this paper is only a modest knowledge of AECs (for example Chapters 4 and 8 of \cite{baldwinbook09}) although we rely on (as black boxes) several facts and definitions from the recent literature (especially \cite{vandieren-symmetry-apal, vandieren-chainsat-apal, vv-symmetry-transfer-afml}). To understand the applications in Section \ref{applications-sec}, a more solid background (described in the papers referenced there) including a good knowledge of the work of the author (from which we quote extensively) is needed.

This paper was written while working on a Ph.D.\ thesis under the direction of Rami Grossberg at Carnegie Mellon University and the author would like to thank Professor Grossberg for his guidance and assistance in his research in general and in this work specifically.

\section{Extracting strict indiscernibles}

Everywhere in this paper, $\K$ denotes a fixed AEC (not necessarily satisfying amalgamation or no maximal models). We assume that the reader is familiar with the definitions of amalgamation, no maximal models, Galois types, and (Galois) saturation. We will use the notation from the preliminaries of \cite{sv-infinitary-stability-afml}.

In particular, $\gtp (\bb / A; N)$ denotes the Galois types of $\bb$ over the set $A$ as computed inside $N$ (so we make use of Galois types over sets, defined as for Galois types over models; note also that the definition does not assume amalgamation). We let $\gS^\alpha (A; N)$ denote the set of all Galois types of sequences of length $\alpha$ over $A$ computed in $N$, and let $\gS^\alpha (M)$ denote the set of all Galois types of sequences of length $\alpha$ over $M$ (computed in any extension $N$ of $M$). When $\alpha = 1$, we omit it.

When working with EM models, we will use the notation from \cite[Chapter IV]{shelahaecbook}:

\begin{defin}\cite[Definition IV.0.8]{shelahaecbook}
  For $\mu \ge \LS (\K)$, let $\Upsilon_{\mu}[\K]$ be the set of $\Phi$ proper for linear orders (that is, $\Phi$ is a set $\{p_n : n < \omega\}$, where $p_n$ is an $n$-variable quantifier-free type in a fixed vocabulary $\tau (\Phi)$ and the types in $\Phi$ can be used to generate a $\tau (\Phi)$-structure $\EM (I, \Phi)$ for each linear order $I$; that is, $\EM (I, \Phi)$ is the closure under the functions of $\tau (\Phi)$ of the universe of $I$ and for any $i_0 < \ldots < i_{n - 1}$ in $I$, $i_0 \ldots i_{n - 1}$ realizes $p_n$) with:

  \begin{enumerate}
  \item $|\tau (\Phi)| \le \mu$.
  \item If $I$ is a linear order of cardinality $\lambda$, $\EM_{\tau (\K)} (I, \Phi) \in \K_{\lambda + |\tau (\Phi)| + \LS (\K)}$, where $\tau (\K)$ is the vocabulary of $\K$ and $\EM_{\tau (\K)} (I, \Phi)$ denotes the reduct of $\EM (I, \Phi)$ to $\tau (\K)$. Here we are implicitly also assuming that $\tau (\K) \subseteq \tau (\Phi)$.
  \item For $I \subseteq J$ linear orders, $\EM_{\tau (\K)} (I, \Phi) \lea \EM_{\tau (\K)} (J, \Phi)$.
  \end{enumerate}

  We call $\Phi$ as above an \emph{EM blueprint}.
\end{defin}

The following follows from Shelah's presentation theorem. We will use it without explicit mention.

\begin{fact}\label{em-existence-fact}
  Let $\mu \ge \LS (\K)$. $\K$ has arbitrarily large models if and only if $\Upsilon_{\mu}[\K] \neq \emptyset$.
\end{fact}

The next notions (due to Shelah) generalize the concept of an indiscernible sequence in a first-order theory. We prefer not to work inside a monster model (one reason is that some of our application will assume only weak versions of amalgamation, e.g.\ the Shelah-Villaveces context \cite{shvi635}), so give more localized definitions here (but assuming a monster model the definitions below coincide with Shelah's).

\begin{defin}[Indiscernibles, Definition 4.1 in \cite{sh394}]
  Let $\K$ be an AEC. Let $N \in \K$. Let $\alpha$ be a non-zero cardinal, $\theta$ be an infinite cardinal, and let $\seq{\ba_i : i < \theta}$ be a sequence of distinct elements with $\ba_i \in \fct{\alpha}{|N|}$ for all $i < \theta$. Let $A \subseteq |N|$ be a set.

  \begin{enumerate}
  \item We say that \emph{$\seq{\ba_i : i < \theta}$ is indiscernible over $A$ in $N$} if for every $n < \omega$, every $i_0 < \ldots < i_{n - 1} < \theta$, $j_0 < \ldots < j_{n - 1} < \theta$, $\gtp (\ba_{i_0} \ldots \ba_{i_n} / A; N) = \gtp (\ba_{j_0} \ldots \ba_{j_n} / A; N)$. When $A = \emptyset$, we omit it and just say that $\seq{\ba_i : i < \theta}$ is indiscernible in $N$.
  \item We say that $\seq{\ba_i : i < \theta}$ is \emph{strictly indiscernible in $N$} if there exists an EM blueprint $\Phi$ (whose vocabulary is allowed to have arbitrary size) and a map $f$ so that, letting $N' := \EM_{\tau (\K)} (\theta, \Phi)$:

    \begin{enumerate}
    \item $f : N \rightarrow N'$ is a $\K$-embedding. For $i < \theta$, let $\bb_i := f (\ba_i)$.
    \item If for $i < \theta$, $\bb_i = \seq{b_{i, j} : j < \alpha}$, then for all $j < \alpha$ there exists a unary $\tau (\Phi)$-function symbol $\rho_j$ such that for all $i < \theta$, $b_{i, j} = \rho_j^{N'} (i)$.
    \end{enumerate}
  \item Let $A \subseteq |N|$. We say that $\seq{\ba_i : i < \theta}$ is \emph{strictly indiscernible over $A$ in $N$} if there exists an enumeration $\ba$ of $A$ such that $\seq{\ba_i \ba : i < \theta}$ is strictly indiscernible in $N$.
  \end{enumerate}
\end{defin}

Because the compactness theorem is not available, indiscernible sequences may in general fail to extend to arbitrarily length. The point of strict indiscernibles is to correct that defect:

\begin{fact}\label{strict-indisc-fact}
  Assume that $\BI := \seq{\ba_i : i < \theta}$ is strictly indiscernible over $A$ in $N$. Then:

  \begin{enumerate}
  \item\label{strict-indisc-1} $\BI$ is indiscernible over $A$ in $N$.
  \item\label{strict-indisc-2} For every $\theta' \ge \theta$, there exists $N' \gea N$ and $\seq{\ba_i : i \in \theta' \backslash \theta}$ such that $\seq{\ba_i : i < \theta'}$ is strictly indiscernible over $A$ in $N'$.
  \end{enumerate}
\end{fact}
\begin{proof}[Proof sketch] \
  \begin{enumerate}
  \item Because $I$ is indiscernible (in the first-order sense, in the vocabulary $\tau (\Phi)$) inside $\EM (I, \Phi)$, and this transfers to Galois types in $\K$.
  \item Use the (first-order) compactness theorem on the theory of the $\EM (\theta, \Phi)$, expanded with constant symbols for the sequence witnessing the strict indiscernibility.
  \end{enumerate}
\end{proof}

The following fact is key to all the subsequent results. It shows that inside an EM model (generated by an ordinal), one can extract a \emph{strictly} indiscernible subsequence from any long-enough sequence. It is due to Shelah and appears as \cite[Claim 4.15]{sh394}. For the convenience of the reader, we give a full proof. 

\begin{fact}\label{main-fact}
  Let $\K$ be an AEC with arbitrarily large models and let $\LS (\K) < \theta \le \lambda$ be cardinals with $\theta$ regular. Let $\kappa < \theta$ be a (possibly finite) cardinal. Let $\Phi \in \Upsilon_{\LS (\K)}[\K]$ be an EM blueprint for $\K$.

  Let $N := \EM_{\tau (\K)} (\lambda, \Phi)$. Let $M \in \K_{\le \LS (\K)}$ be such that $M \lea N$. Let $\seq{\ba_i : i < \theta}$ be a sequence of distinct elements such that for all $i < \theta$, $\ba_i \in \fct{\kappa}{|N|}$.

  If $\theta_0^\kappa < \theta$ for all $\theta_0 < \theta$, then there exists $w \subseteq \theta$ with $|w| = \theta$ such that $\seq{\ba_i : i \in w}$ is strictly indiscernible over $M$ in $N$.
\end{fact}
\begin{remark}
  We do \emph{not} assume amalgamation (we will work entirely inside $\EM_{\tau (\K)} (\lambda, \Phi)$).
\end{remark}
\begin{remark}
  The main case for us is $\kappa < \aleph_0$, where the cardinal arithmetic assumption holds trivially and the proof is simpler.
\end{remark}
\begin{remark}
  We are assuming that $\|M\| \le \LS (\K)$ only to simplify the notation: if $\mu := \|M\| \in (\LS (\K), \theta)$, we can just replace $\K$ by $\K_{\ge \mu}$.
\end{remark}
\begin{proof}[Proof of Fact \ref{main-fact}]
  First we claim that one can assume without loss of generality that $\kappa < \LS (\K)$. Assume that the statement of the lemma has been proven for that case. If $\kappa > \LS (\K)$ one can replace $\K$ with $\K_{\ge \kappa}$ (and increase $M$) so assume that $\kappa \le \LS (\K)$. Now if $\kappa = \LS (\K)$, then $2^{\LS (\K)} = \kappa^\kappa < \theta$ so we can replace $\K$ by $\K_{\ge \LS (\K)^+}$ and work there. Thus assume without loss of generality that $\kappa < \LS (\K)$.
  
  Pick $u \subseteq \lambda$ such that $|u| = \theta$, $M \lea  N_0 := \EM_{\tau (\K)} (u, \Phi)$, and $\ba_i \in \fct{\kappa}{|N_0|}$ for all $i < \theta$. Increasing $M$ if necessary, we can assume without loss of generality that $M = \EM_{\tau (\K)} (u', \Phi)$ for some $u' \subseteq u$ with $|u'| = \LS (\K)$.

  For each $i < \theta$, we can also pick $u_i \subseteq u$ with $|u_i| < \kappa^+ + \aleph_0$ such that $\ba_i \in \fct{\kappa}{|\EM_{\tau (\K)} (u_i, \Phi)|}$. Without loss of generality $u = u' \cup \bigcup_{i < \theta} u_i$. By the pigeonhole principle, we can without loss of generality fix an ordinal $\alpha < \kappa^+ + \aleph_0$ such that $\otp (u_i) = \alpha$ for all $i < \theta$. List $u_i$ in increasing order as $\bu_i := \seq{u_{i, j} : j < \alpha}$. By pruning further (using that $\LS (\K)^\kappa < \theta$), we can assume without loss of generality that for each $i, i' < \theta$ and $j < \alpha$, the $u'$-cut of $u_{i, j}$ and $u_{i', j}$ are the same (i.e.\ for any $\gamma \in u'$, $\gamma < u_{i, j}$ if and only if $\gamma < u_{i', j}$).

  Pruning again with the $\Delta$-system lemma, we can assume without loss of generality that $\seq{u_i : i < \theta}$ forms a $\Delta$-system (see Definition II.1.4 and Theorem II.1.6 in \cite{kunenbook}; at that point we are using that $\theta_0^\kappa < \theta$ for all $\theta_0 < \theta$). All these pruning steps ensure that $\seq{\bu_i : i < \theta}$ is indiscernible over $u'$ in the vocabulary of linear orders.

  Now list $\ba_i$ as $\seq{a_{i, j} : j < \kappa}$. Fix $i < \theta$. Since $\ba_i \in \fct{\kappa}{|\EM_{\tau (\K)} (u_i)|}$, for each $j < \kappa$ there exists a $\tau (\Phi)$-term $\rho_{i, j}$ of arity $n := n_{i, j}$ and $j_0^{i, j} < \ldots < j_{n - 1}^{i,j} < \alpha$ such that $a_{i, j} = \rho_{i, j} \left(u_{i, j_0^{i,j}} \ldots u_{i, j_{n - 1}^{i, j}}\right)$. By the pigeonhole principle applied to the map $i \mapsto \seq{(\rho_{i, j}, n_{i, j}, j_0^{i,j}, \ldots, j_{n_{i, j} - 1}^{i, j}) : j < \kappa}$ (using that $\LS (\K)^\kappa < \theta$), we can assume without loss of generality that these depend only on $j$, i.e.\ $\rho_{i, j} = \rho_j$, $n_{i, j} = n_j$, and $j_\ell^{i, j} = j_\ell^j$.

  Let $\bu'$ be an enumeration of $u'$, and let $\ba_i' := \ba_i \bu'$. We are assuming that $\kappa < \LS (\K)$ so $\ell (\ba_i') < \LS (\K)$. Let $\bb_i$ be $\ba_i'$ followed by $a_{i, 0}$ repeated $\LS (\K)$-many times (we only do this to make the order type of each element of our sequence $\LS (\K)$ and hence simplify the notation). Then $\ell (\bb_i) = \LS (\K)$. As before, say $\bb_i = \seq{b_{i, j} : j < \LS (\K)}$. Let $u_i' := u_i \cup u'$. Let $\bu_i'$ be an enumeration of $u_i'$ of type $\LS (\K)$ (so not necessarily increasing). Say $\bu_i' = \seq{u_{i,j}' : j < \LS (\K)}$.

  By the pruning carried out previously and the definition of $u'$, we have that for each $i < \theta$ and each $j < \LS (\K)$, there exists a $\tau (\Phi)$-term $\rho_j$ of arity $n_j$ and $j_0^j < \ldots < j_{n - 1}^j < \LS (\K)$ (the point is that they do \emph{not} depend on $i$) such that $b_{i,j} = \rho_j (u_{i, j_0^j}', \ldots, u_{i, j_{n_j - 1}^j}')$. We will build an EM blueprint $\Psi$ witnessing that $\seq{\bb_i : i < \theta}$ is strictly indiscernible.

  For each $n$-ary $\tau (\Phi)$-term $\rho$, each $\gamma_0 < \ldots < \gamma_{n - 1} < \LS (\K)$, and each $i_0 < \ldots < i_{n - 1} < \theta$, we define a function $g = g_{\rho, \gamma_0 \ldots \gamma_{n - 1}, i_0 \ldots i_{n - 1}}$ as follows: for $i < \theta$, $j < \LS (\K)$, let $g (u_{i, j}') := \rho (u_{i + i_0, j + \gamma_0}', \ldots, u_{i + i_{n - 1}, j + \gamma_{n - 1}}')$. We naturally extend $g$ to have domain $N_0 = \EM (u, \Phi)$ (recall from the beginning of the proof that $u = u' \cup \bigcup_{i < \theta} u_i$). The vocabulary of $\Psi$ will consist of the vocabulary of $\tau (\Phi)$ together with a unary function symbol for each $g_{\rho, \gamma_0 \ldots \gamma_{n - 1}, i_0 \ldots i_{n - 1}}$. For $n < \omega$, let $p_n := \tp_{\tau (\Psi)} (u_{0, 0}' u_{1, 0}' \ldots u_{n - 1, 0}' / \emptyset; N_0)$ and let $\Psi := \{p_n : n < \omega\}$. Then $\Psi$ is as desired.
\end{proof}

We will use this fact to study lengths of a variation of the order property from the first-order context defined in terms of Galois types. We again give a local definition (as in \cite{sv-infinitary-stability-afml}), but the global version is due to Shelah.

\begin{defin}[Order property, Definition 4.3 in \cite{sh394}]\label{op-def} \
  \begin{enumerate}
  \item Let $N \in \K$. We say that \emph{$N$ has the $(\kappa, \mu)$-order property of length $\theta$} if there exists a sequence $\seq{\ba_i : i < \theta}$ and a set $A$ such that $\ba_i \in \fct{\kappa}{|N|}$ for every $i < \theta$, $A \subseteq |N|$, $|A| \le \mu$, and for every $i_0 < i_1 < \theta$, $j_0 < j_1 < \theta$, $\gtp (\ba_{i_0} \ba_{i_1} / A; N) \neq \gtp (\ba_{j_1} \ba_{j_0} / A; N)$.
  \item We say that \emph{$\K$ has the $(\kappa, \mu)$-order property of length $\theta$} if some $N \in \K$ has it.
  \item We say that \emph{$\K$ has the $(\kappa, \mu)$-order property} if it has the $(\kappa,\mu)$-order property of length $\theta$ for all cardinals $\theta$.
  \item When $\mu = 0$, we omit it and talk of the $\kappa$-order property.
  \end{enumerate}
\end{defin}
\begin{remark}
  For $T$ a first-order theory and $\K$ its corresponding AEC of models, the following are equivalent:

  \begin{enumerate}
  \item $T$ is unstable.
  \item $\K$ has the $(\kappa, 0)$-order property, for some $\kappa < \aleph_0$.
  \item $\K$ has the $(\kappa, \mu)$-order property, for some cardinals $\kappa$ and $\mu$.
  \end{enumerate}
\end{remark}

An easy consequence of Fact \ref{main-fact} is that if a long-enough order property holds, then we can assume that the sequence witnessing it is strictly indiscernible, and hence extend it:

\begin{thm}\label{op-local-thm}
  Let $\K$ be an AEC with arbitrarily large models and let $\LS (\K) < \lambda$. Let $\kappa < \lambda$ be a (possibly finite) cardinal. Let $\Phi \in \Upsilon_{\LS (\K)}[\K]$ be an EM blueprint for $\K$.

  Let $N := \EM_{\tau (\K)} (\lambda, \Phi)$. If $N$ has the $(\kappa, \LS (\K))$-order property of length $\left(\LS (\K)^\kappa\right)^+$ and $\LS (\K)^\kappa < \lambda$, then $\K$ has the $(\kappa, \LS (\K))$-order property (of \emph{any} length).
\end{thm}
\begin{proof}
  Set $\theta := \left(\LS (\K)^\kappa\right)^+$. Fix $\seq{\ba_i : i < \theta}$ and $A$ witnessing that $N$ has the $(\kappa, \LS (\K))$-order property of length $\theta$. Using the Löwenheim-Skolem-Tarski axiom, pick $M \in \K_{\LS (\K)}$ such that $A \subseteq |M|$ and $M \lea N$. By Fact \ref{main-fact}, there exists $w \subseteq \theta$ such that $|w| = \theta$ and $\seq{\ba_i : i \in w}$ is strictly indiscernible over $M$ in $N$. Without loss of generality, $w = \theta$. Let $\theta' \ge \theta$ be an arbitrary cardinal. By Fact \ref{strict-indisc-fact}.(\ref{strict-indisc-2}), we can find $N' \gea N$ and $\seq{\ba_i : i \in \theta' \backslash \theta}$ such that $\BI := \seq{\ba_i : i < \theta'}$ is strictly indiscernible over $M$ in $N'$. By Fact \ref{strict-indisc-fact}.(\ref{strict-indisc-1}), $\BI$ is indiscernible over $M$ in $N'$. We claim that $\BI$ witnesses that $N'$ has the $(\kappa, \LS (\K))$-order property of length $\theta'$. Indeed, if $i_0 < i_1 < \theta'$, $j_0 < j_1 < \theta'$, then by indiscernibility, $p := \gtp (\ba_{i_0} \ba_{i_1} / M; N') = \gtp (\ba_0 \ba_1 / M; N')$ and $q := \gtp (\ba_{j_1} \ba_{j_0} / M; N') = \gtp (\ba_1 \ba_0 / M; N')$. Because the original sequence $\seq{\ba_i : i < \theta}$ was witnessing the $(\kappa, \LS (\K))$-order property, we have that $p \neq q$, as desired. 
\end{proof}
\begin{remark}
  By appending an enumeration of the base set to each element of the sequence, we get that the $(\kappa, \mu)$-order property implies the $(\kappa + \mu)$-order property. However Theorem \ref{op-local-thm}  applies more easily to the $(\kappa, \mu)$-order property: think for example of the case $\kappa < \aleph_0$, when we always have that $\LS (\K)^\kappa = \LS (\K) < \lambda$.
\end{remark}

\section{Solvability and failure of the order property}

We recall Shelah's definition of solvability \cite[Definition IV.1.4]{shelahaecbook}, and mention a more convenient notation for it with only one cardinal parameter. We also introduce semisolvability, which only asks for the EM model to be universal (instead of superlimit). Both variations are equivalent to superstability in the first-order case (see \cite{shvi-notes-v3-toappear} and \cite[Corollary 5.3]{gv-superstability-v5-toappear}). Shelah writes that solvability is perhaps the true analog of superstability in abstract elementary classes \cite[N\S4(B)]{shelahaecbook}.

\begin{defin}\label{solv-def}
  Let $\LS (\K) \le \mu \le \lambda$.
  \begin{enumerate}
  \item $M \in \K$ is \emph{universal in $\lambda$} if $M \in \K_\lambda$ and for any $N \in \K_{\lambda}$ there exists $f: N \rightarrow M$.
  \item \cite[Definition IV.0.5]{shelahaecbook} $M \in \K$ is \emph{superlimit in $\lambda$} if:
    \begin{enumerate}
    \item $M$ is universal in $\lambda$.
    \item $M$ has a proper extension.
    \item For any limit ordinal $\delta < \lambda^+$ and any increasing continuous chain $\seq{M_i : i \le \delta}$ in $\K_{\lambda}$, if $M \cong M_i$ for all $i < \delta$, then $M \cong M_\delta$.
    \end{enumerate}
  \item \cite[Definition IV.1.4.(1)]{shelahaecbook} We say that \emph{$\Phi$ witnesses $(\lambda, \mu)$-solvability} if:
      \begin{enumerate}
        \item $\Phi \in \Upsilon_{\mu}[\K]$.
        \item If $I$ is a linear order of size $\lambda$, then $\EM_{\tau (\K)} (I, \Phi)$ is superlimit in $\lambda$.
      \end{enumerate}
    \item \emph{$\Phi$ witnesses $(\lambda, \mu)$-semisolvability} if:
      \begin{enumerate}
      \item $\Phi \in \Upsilon_{\mu}[\K]$
      \item If $I$ is a linear order of size $\lambda$, then $\EM_{\tau (\K)} (I, \Phi)$ is universal in $\lambda$.
      \end{enumerate}
    \item $\K$ is \emph{$(\lambda, \mu)$-[semi]solvable} if there exists $\Phi$ witnessing $(\lambda, \mu)$-[semi]solvability.
    \item $\K$ is \emph{$\lambda$-[semi]solvable} (or \emph{[semi]solvable in $\lambda$}) if $\K$ is $(\lambda, \LS (\K))$-[semi]solvable.
  \end{enumerate}
\end{defin}
\begin{remark}\label{superlimit-uq}
  By a straightforward argument (similar to the proof of Corollary \ref{solvable-saturated-cor}.(\ref{solvable-sat-2}) here), superlimit models must be unique.
\end{remark}
\begin{remark}\label{solv-rmk}
  Because superlimit models are universal, if $\K$ is $(\lambda, \mu)$-solvable, then $\K$ is $(\lambda, \mu)$-semisolvable. Also, the model in a categoricity cardinal must be superlimit (if it has a proper extension), so if $\K$ has arbitrarily large models and is categorical in $\lambda \ge \LS (\K)$, then $\K$ is solvable in $\lambda$.
\end{remark}

The reader not especially interested in solvability can simply remember the last remark and read ``categorical'' instead of ``solvable'' whenever appropriate.

We can combine Theorem \ref{op-local-thm} with semisolvability (we are still not assuming amalgamation):

\begin{thm}\label{op-thm}
  Let $\lambda > \LS (\K)$. If $\K$ is semisolvable in $\lambda$, then for any cardinals $\mu, \kappa$, and any $M \in \K_\lambda$, $M$ does \emph{not} have the $(\kappa, \mu)$-order property of length $\left((\mu + \LS (\K))^{\kappa}\right)^+$.
\end{thm}
\begin{remark}\label{rmk-op-thm-1}
  The statement is interesting only when $(\mu + \LS (\K))^{\kappa} < \lambda$, but is still vacuously true otherwise (no $M$ of size $\lambda$ can witness an order property of length longer than $\lambda$). The main case for us is $\kappa < \aleph_0$, $\mu \in [\LS (\K), \lambda)$, where the result tells us that the $(\kappa, \mu)$ order property of length $\mu^+$ must fail.
\end{remark}
\begin{remark}
  A similar result is \cite[Claim IV.1.5.(2)]{shelahaecbook}. There the conclusion is weaker (only the $(\kappa, \mu)$-order property fails, nothing is said about the length), and the hypothesis uses solvability instead of semisolvability. The proof relies on Shelah's construction of many models from the order property.
\end{remark}
\begin{proof}[Proof of Theorem \ref{op-thm}]
  Replacing $\mu$ by $(\mu + \LS (\K))^\kappa$ if necessary, we can assume without loss of generality that $\mu \ge \LS (\K)$ and $\mu = \mu^{\kappa}$. If $\mu \ge \lambda$ the result is vacuously true (see Remark \ref{rmk-op-thm-1}) so assume without loss of generality that $\mu < \lambda$. Replacing $\K$ by $\K_{\ge \mu}$ if necessary, we can also assume without loss of generality that $\mu = \LS (\K)$.

  \underline{Claim 1}: For any $M \in \K_\lambda$ and any $A \subseteq |M|$, if $|A| = \mu$, then $|\gS^{\kappa} (A; M)| \le \mu$.

  \underline{Proof of Claim 1}: This is essentially Morley's argument from \cite[Theorem 3.7]{morley-cip} but we give a short proof using (the much stronger) Fact \ref{main-fact}. Fix $M \in \K_\lambda$. Let $\Phi$ be an EM blueprint witnessing semisolvability. By definition of semisolvability, we can embed $M$ inside $\EM_{\tau (\K)} (\lambda, \Phi)$, hence assume without loss of generality that $M = \EM_{\tau (\K)} (\lambda, \Phi)$. Let $A \subseteq |M|$ with $|A| = \mu$ and let $M_0 \in \K_\mu$ be such that $A \subseteq |M_0|$ and $M_0 \lea M$. Let $\seq{\ba_i : i < \mu^+}$ be an arbitrary sequence of elements from $\fct{\kappa}{|M|}$. By Fact \ref{main-fact} (with $\LS (\K)$, $\theta$, $\lambda$, $M$, $N$ there standing for $\mu$, $\mu^+$, $\lambda$, $M_0$, $M$ here), we have (in particular) that there exists $i < j < \mu^+$ such that $\gtp (\ba_i / M_0; M) = \gtp (\ba_j / M_0; M)$. Since $\seq{\ba_i : i < \mu^+}$ was arbitrary, this shows that $|\gS^\kappa (M_0; M)| \le \mu$ and hence $|\gS^{\kappa} (A; M)| \le \mu$, as desired. $\dagger_{\text{Claim 1}}$
  
  \underline{Claim 2}: If there exists $M \in \K_\lambda$ such that $M$ has the $(\kappa, \mu)$-order property of length $\mu^+$, then $\K$ has the $(\kappa, \mu)$-order property.

  \underline{Proof of Claim 2}: Fix $M \in \K_\lambda$ with the $(\kappa, \mu)$-order property of length $\mu^+$. Let $\Phi$ be an EM blueprint witnessing semisolvability. By definition of semisolvability, we can embed $M$ inside $\EM_{\tau (\K)} (\lambda, \Phi)$, hence $\EM_{\tau (\K)} (\lambda, \Phi)$ has the $(\kappa, \mu)$-order property of length $\mu^+$. Now apply Theorem \ref{op-local-thm}. $\dagger_{\text{Claim 2}}$

  \underline{Claim 3}: If $\K$ has the $(\kappa, \mu)$-order property, then there exists $M \in \K_\lambda$ and $A \subseteq |M|$ such that $|A| = \mu$ but $|\gS^\kappa (A; M)| > \mu$.

  \underline{Proof of Claim 3}: We follow the proof of \cite[Fact 5.13]{bgkv-apal} (whose statement was first observed by Shelah \cite[Claim 4.7.(2)]{sh394}). Let $I_0 \subseteq I$ be linear orders such that $|I_0| = \mu$, $|I| > \mu$, and $I_0$ is dense in $I$ (for example, take $\chi$ least such that $\mu < \mu^{\chi}$, and let $I_0$ be $\mu^{<\chi}$ ordered lexicographically and $I$ be $\mu^{\le \chi}$, also ordered lexicographically). Without loss of generality, $|I| = \mu^+$. Let $J \supseteq I$ be such that $|J| = \lambda$. Combining Shelah's presentation theorem with Morley's method (using that $\K$ has the $(\kappa, \mu)$-order property), we can get models $M_0 \lea M$ with $M_0 \in \K_\mu$, $M \in \K_\lambda$, such that $M$ contains a sequence $\seq{\ba_i : i \in J}$ with $\ell (\ba_i) = \kappa$ and $i_0 < j_0$, $i_1 < j_1$ implies $\gtp (\ba_{i_0} \ba_{i_1} / M_0; M) \neq \gtp (\ba_{j_1} \ba_{i_1} / M_0; M)$. Let $A := |M_0| \cup \bigcup\{\ran{\ba_i} : i \in I_0\}$ and observe that for $i < j$ both in $I$, $\gtp (\ba_i / A; M) \neq \gtp (\ba_j / A; M)$. Indeed, by density there is $k \in I_0$ such that $i < k < j$ and so by construction $\gtp (\ba_i \ba_k / M_0; M) \neq \gtp (\ba_j \ba_k / M_0; M)$. Therefore $|\gS^{\kappa} (A; M)| \ge |I| = \mu^+ > \mu$. $\dagger_{\text{Claim 3}}$

  We can now conclude the proof of Theorem \ref{op-thm}. If there were a model of size $\lambda$ with the $(\kappa, \mu)$-order property of length $\mu^+$, then Claim 2 and 3 together would contradict Claim 1. Therefore there is no such model, as desired.
\end{proof}

\section{Solvability and saturation}

In this section, we prove Theorem \ref{sat-abstract} from the abstract (the model in the categoricity cardinal is saturated). We will rely on the following local version of superstability, already implicit in \cite{shvi635} and since then studied in many papers, e.g.\ \cite{vandierennomax, gvv-mlq, indep-aec-apal, bv-sat-afml, gv-superstability-v5-toappear, vandieren-symmetry-apal}. We quote the definition from \cite[Definition 10.1]{indep-aec-apal}:

\begin{defin}\label{ss-def}
$\K$ is \emph{$\mu$-superstable} (or \emph{superstable in $\mu$})
if:

  \begin{enumerate}
    \item $\mu \ge \LS (\K)$.
    \item $\K_\mu$ is nonempty, has joint embedding, amalgamation, and no maximal models.
    \item $\K$ is stable in $\mu$.
    \item There are no long splitting chains in $\mu$:

      For any limit ordinal $\delta < \mu^+$, for every sequence $\langle M_i\mid i<\delta\rangle$ of
      models of cardinality $\mu$ with $M_{i+1}$ universal over $M_i$ and for every $p\in\gS(\bigcup_{i < \delta} M_i)$, there exists $i<\delta$ such that $p$ does not $\mu$-split over $M_i$.
\end{enumerate}
\end{defin}

We will also use the concept of symmetry for splitting isolated in \cite{vandieren-symmetry-apal}:

\begin{defin}\label{sym-def}
For $\mu \ge \LS (\K)$, we say that $\K$ has \emph{$\mu$-symmetry} (or \emph{symmetry in $\mu$}) if  whenever models $M,M_0,N\in\K_\mu$ and elements $a$ and $b$  satisfy the conditions (\ref{limit sym cond})-(\ref{last}) below, then there exists  $M^b$  a limit model over $M_0$, containing $b$, so that $\gtp(a/M^b)$ does not $\mu$-split over $N$.
\begin{enumerate} 
\item\label{limit sym cond} $M$ is universal over $M_0$ and $M_0$ is a limit model over $N$.
\item\label{a cond}  $a\in |M|\backslash |M_0|$.
\item\label{a non-split} $\gtp(a/M_0)$ is non-algebraic and does not $\mu$-split over $N$.
\item\label{last} $\gtp(b/M)$ is non-algebraic and does not $\mu$-split over $M_0$. 
\end{enumerate}
\end{defin}
\begin{remark}
  We will only use the consequences of Definitions \ref{ss-def} and \ref{sym-def}, not their exact content.
\end{remark}

By an argument of Shelah and Villaveces \cite[Theorem 2.2.1]{shvi635} (see also \cite{shvi-notes-v3-toappear}), superstability holds below the categoricity (or just semisolvability) cardinal:

\begin{fact}[The Shelah-Villaveces Theorem]\label{shvi}
  Let $\lambda > \LS (\K)$. Assume that $\K_{<\lambda}$ has amalgamation and no maximal models. If $\K$ is semisolvable in $\lambda$, then $\K$ is superstable in any $\mu \in [\LS (\K), \lambda)$.
\end{fact}
\begin{remark}
  Here and below, we are assuming amalgamation and no maximal models but only (strictly) below $\lambda$. In at least one case ($\lambda \ge \beth_{\left(2^{\chi}\right)^+}$ where $\chi > \LS (\K)$ is a measurable cardinal \cite{kosh362, tamelc-jsl}), these assumptions are known to follow (inside $\K_{\ge \chi}$ for the example just mentioned) from categoricity in $\lambda$ but they are \emph{not} known to hold above $\lambda$.
\end{remark}

A slight improvement on \cite[Theorem 5.8]{vv-symmetry-transfer-afml} shows that failure of symmetry implies an appropriate order property:

\begin{fact}\label{op-sym-fact}
  Let $\lambda > \mu \ge \LS (\K)$. Assume that $\K$ is superstable in every $\chi \in [\mu, \lambda)$. If $\K$ does \emph{not} have $\mu$-symmetry, then $\K$ has the $(2, \mu)$-order property of length $\lambda$ (recall Definition \ref{op-def}).
\end{fact}
\begin{proof}
  The proof of \cite[Theorem 5.8]{vv-symmetry-transfer-afml} builds $M_0 \in \K_\mu$, elements $a, b$, an increasing continuous sequence $\seq{N_\alpha : \alpha < \lambda}$, and sequences $\seq{a_\alpha : \alpha < \lambda}$, $\seq{b_\alpha : \alpha < \lambda}$ such that for all $\alpha, \beta < \lambda$ (for the last two conditions, see the ``This is enough'' part of the proof of \cite[Theorem 5.8]{vv-symmetry-transfer-afml}):

  \begin{enumerate}
  \item $N_\alpha \in \K_{\mu + |\alpha|}$.
  \item $a, b \in |N_0|$.
  \item $M_0 \lea N_0$.
  \item $a_\alpha, b_\alpha \in |N_{\alpha + 1}|$.
  \item\label{important-1} If $\beta < \alpha$, $\gtp (a b / M_0; N_{\alpha + 1}) \neq \gtp (a_\alpha b_\beta / M_0; N_{\alpha + 1})$.
  \item\label{important-2} If $\beta \ge \alpha$, $\gtp (a b / M_0; N_{\beta + 1}) = \gtp (a_\alpha b_\beta / M_0; N_{\beta + 1})$.
  \end{enumerate}

  Let $N_\lambda := \bigcup_{\alpha < \lambda} N_\alpha$ and for $\alpha < \lambda$, let $\bc_\alpha := a_\alpha b_\alpha$. Then by (\ref{important-1}) and (\ref{important-2}), $\seq{\bc_\alpha : \alpha < \lambda}$ and the set $A := |M_0|$ witness that $N_\lambda$ (and hence $\K$) has the $(2, \mu)$-order property of length $\lambda$.
\end{proof}
\begin{remark}
  We have \emph{not} explicitly assumed amalgamation and no maximal models, as this is implied (at the relevant cardinals) by the definition of superstability.
\end{remark}

We conclude that $\mu$-symmetry follows from categoricity (or just semisolvability) in some $\lambda > \mu$. This improves on \cite[Corollary 7.2]{vv-symmetry-transfer-afml} which asked for the model of cardinality $\lambda$ to be $\mu^+$-saturated (we will see next that this saturation also follows).

\begin{cor}\label{sym-cor}
  Let $\lambda > \LS (\K)$. Assume that $\K_{<\lambda}$ has amalgamation and no maximal models. If $\K$ is semisolvable in $\lambda$, then for any $\mu \in [\LS (\K), \lambda)$, $\K$ has $\mu$-symmetry.
\end{cor}
\begin{proof}
  By Fact \ref{shvi}, $\K$ is superstable in any $\mu \in [\LS (\K), \lambda)$. Fix such a $\mu$. Suppose for a contradiction that $\K$ does not have $\mu$-symmetry. By Fact \ref{op-sym-fact}, $\K$ has the $(2, \mu)$-order property of length $\lambda$. In particular, $\K$ has the $(2, \mu)$-order property of length $\mu^+$. This contradicts Theorem \ref{op-thm} (where $\kappa$ there stands for $2$ here).
\end{proof}

We will make strong use of the relationship between symmetry and chains of saturated models (due to VanDieren):

\begin{fact}[Theorem 1 in \cite{vandieren-chainsat-apal}]\label{sat-sym-fact}
  If $\K$ is $\mu$-superstable, $\mu^+$-superstable, and has $\mu^+$-symmetry, then the union of any increasing chain of $\mu^+$-saturated models is $\mu^+$-saturated.
\end{fact}

We have arrived to Theorem \ref{sat-abstract} from the abstract. We first prove a lemma:

\begin{lem}\label{saturated-lem}
  Let $\lambda > \LS (\K)$. If for every $\mu \in [\LS (\K), \lambda)$, $\K$ is $\mu$-superstable and has $\mu$-symmetry, then $\K$ has a saturated model of cardinality $\lambda$.
\end{lem}
\begin{proof}
  If $\lambda$ is a successor, then we can build the desired model using stability below $\lambda$, so assume that $\lambda$ is limit.

    Let $\delta := \cf{\lambda}$. Fix an increasing sequence $\seq{\mu_i : i < \delta}$ cofinal in $\lambda$ such that $\LS (\K) \le \mu_0$. We build an increasing chain $\seq{M_i : i < \delta}$ in $\K_\lambda$ such that for all $i < \delta$, $M_i$ is $\mu_i^+$-saturated. This is enough since then it is easy to check that $\bigcup_{i < \delta} M_i$ is saturated. This is possible: Using Fact \ref{sat-sym-fact}, for any $i < \delta$, any union of an increasing chain of $\mu_i^+$-saturated models is $\mu_i^+$-saturated (note that $\mu_i^+ < \lambda$ as $\lambda$ is limit). Thus it is straightforward to carry out the construction.  
\end{proof}

\begin{cor}\label{solvable-saturated-cor}
  Let $\lambda > \LS (\K)$. Assume that $\K_{<\lambda}$ has amalgamation and no maximal models.

  \begin{enumerate}
  \item\label{solvable-sat-1} If $\K$ is semisolvable in $\lambda$, then $\K$ has a saturated model of cardinality $\lambda$.
  \item\label{solvable-sat-2} If $\Phi$ is an EM blueprint witnessing that $\K$ is solvable in $\lambda$, then for any linear order $I$ of cardinality $\lambda$, $\EM_{\tau (\K)} (I, \Phi)$ is saturated.
  \item\label{solvable-sat-3} If $\K$ has arbitrarily large models and is categorical in $\lambda$, then the model of cardinality $\lambda$ is saturated.
  \end{enumerate}
\end{cor}
\begin{proof} \
  \begin{enumerate}
  \item By Fact \ref{shvi} and Corollary \ref{sym-cor}, $\K$ is superstable and has symmetry in any $\mu \in [\LS (\K), \lambda)$. Now apply Lemma \ref{saturated-lem}.
  \item We show more generally that if $\K$ is semisolvable in $\lambda$ and $M$ is superlimit in $\lambda$, then $M$ is saturated. We build increasing continuous chains $\seq{M_i : i \le \lambda}$, $\seq{N_i : i \le \lambda}$ in $\K_\lambda$ such that for any $i < \lambda$:

  \begin{enumerate}
  \item $M_i \cong M$.
  \item $M_i \lea N_i \lea M_{i + 1}$.
  \item $N_{i + 1}$ is saturated.
  \end{enumerate}

  This is possible by the first part (noting that the saturated model must be universal). This is enough: because $M$ is superlimit, $M \cong M_\lambda = N_\lambda$. Further, $N_\lambda$ must be saturated: if $\lambda$ is a successor this is clear and if $\lambda$ is limit this is because for any $\mu < \lambda$ the union of any increasing chain of $\mu$-saturated models is $\mu$-saturated. Since $N_\lambda$ is saturated, $M$ is also saturated, as desired.
\item By Remark \ref{solv-rmk}, $\K$ is solvable in $\lambda$, so apply the previous part.
  \end{enumerate}
\end{proof}

\section{Applications}\label{applications-sec}

\subsection{Solvability transfers}

We can now prove Corollary \ref{solv-abstract} from the abstract.

\begin{cor}[Downward solvability transfer]\label{solv-downward}
  Let $\lambda > \LS (\K)$. Assume that $\K_{<\lambda}$ has amalgamation and no maximal models. If $\K$ is solvable in $\lambda$, then there exists an EM blueprint $\Psi$ which witnesses that $\K$ is solvable in $\mu$ for any $\mu \in (\LS (\K), \lambda]$.
\end{cor}
\begin{proof}
  Let $\Phi$ be an EM blueprint witnessing that $\K$ is solvable in $\lambda$. By Corollary \ref{solvable-saturated-cor}, $\EM_{\tau (\K)} (J, \Phi)$ is saturated for any linear order $J$ of cardinality $\lambda$. We now use \cite[Subfact 6.8]{sh394} (a full proof is given in \cite{sh394-updated}, the online version of \cite{sh394}). It says that there exists an EM blueprint $\Psi \in \Upsilon_{\LS (\K)}[\K]$ such that:
  \begin{enumerate}
  \item For any linear order $I$ there exists a linear order $J$ with, $\EM_{\tau (\K)} (I, \Psi) = \EM_{\tau (\K)} (J, \Phi)$. In particular, $\Psi$ still witnesses that $\K$ is solvable in $\lambda$.
  \item For any $\mu \in (\LS (\K), \lambda]$ and any linear order $I$ of cardinality $\mu$, $\EM_{\tau (\K)} (I, \Psi)$ is saturated.
  \end{enumerate}

  By Fact \ref{shvi} and Corollary \ref{sym-cor}, $\K$ is superstable and has symmetry in every $\mu \in [\LS (\K), \lambda)$. Now let $\mu \in (\LS (\K), \lambda]$. We want to see that $\Psi$ witnesses solvability in $\mu$. By the above, $\Psi$ witnesses solvability in $\lambda$, so assume that $\mu < \lambda$. Using Fact \ref{sat-sym-fact} it is straightforward to see that the union of any increasing chain of $\mu$-saturated models will be $\mu$-saturated. In other words, the saturated model of cardinality $\mu$ is superlimit and therefore $\Psi$ witnesses that $\K$ is solvable in $\mu$.
\end{proof}
\begin{remark}
  It is natural to ask what happens if $\mu = \LS (\K)$. In that case, if $\Psi$ witnesses solvability in $\LS (\K)^+$ we can find a linear order $J$ of size $\LS (\K)$ such that $\EM_{\tau (\K)} (J, \Psi)$ is limit (see the proof of \cite[Lemma I.6.3]{sh394}). This implies that $\EM_{\tau (\K)} (I \times J, \Psi)$ is limit for any linear order $I$ of size at most $\LS (\K)$ (here $I \times J$ is ordered with the lexicographical ordering). The class of linear orders of the form $I \times J$ is an AEC with arbitrarily large models and hence has an EM blueprint. Composing this blueprint with $\Psi$, we can find a blueprint $\Psi'$ such that $\EM_{\tau (\K)} (I, \Psi') = \EM_{\tau (\K)} (I \times J, \Psi)$ for any linear order $I$. In particular, $\Psi'$ also witnesses solvability in $(\LS (\K), \lambda]$. Moreover, $\EM_{\tau (\K)} (I, \Psi')$ is limit for any linear order $I$ of cardinality $\LS (\K)$. This implies that $\Psi'$ witnesses \emph{semisolvability} in $\LS (\K)$, but it is not clear that the limit model is superlimit (even though it is unique), see \cite[Question 6.12]{vv-symmetry-transfer-afml}. Therefore we do not know if $\Psi'$ witnesses solvability in $\LS (\K)$, but it will if there is any superlimit in $\LS (\K)$.
\end{remark}

Assuming tameness, we can also get an upward transfer. Note that here only semisolvability is assumed so also the downward part of Corollary \ref{upward-solv} is interesting.

\begin{cor}\label{upward-solv}
  Assume that $\K$ is $\LS (\K)$-tame and has amalgamation and no maximal models. Write $\mu_0 := \left(\beth_{\omega + 2} (\LS (\K))\right)^+$. If $\K$ is semisolvable in $\lambda$ for \emph{some} $\lambda > 2^{\LS (\K)}$, then $\K$ is $(\mu, \mu_0)$-solvable for \emph{all} $\mu \ge \mu_0$.
\end{cor}
\begin{remark}
  This improves on the threshold from \cite[Theorem 5.4]{gv-superstability-v5-toappear}: there $\mu_0$ was around $\beth_{\left(2^{\LS (\K)}\right)^+}$. We quote freely from there in the proof.
\end{remark}
\begin{remark}
  In the conclusion, the same blueprint will witness $(\mu, \mu_0)$-solvability for all $\mu$. 
\end{remark}
\begin{proof}[Proof of Corollary \ref{upward-solv}]
  In the proof of \cite[Theorem 5.4]{gv-superstability-v5-toappear}, the only reason for the threshold to be around $\beth_{\left(2^{\LS (\K)}\right)^+}$ was a bound on a cardinal $\chi_0$ so that $\K$ does not have the $\LS (\K)$-order property of length $\chi_0$. Now using Theorem \ref{op-thm}, we get that $\K$ does not have the $\LS (\K)$-order property of length $\left(2^{\LS (\K)}\right)^+$. Following \cite[Section 4]{gv-superstability-v5-toappear}, we obtain that $\K$ is $(\mu, \mu_0)$-solvable for all $\mu \ge \mu_0$.
\end{proof}

\subsection{Structure of categorical AECs with amalgamation}

Directly from existing results and Corollary \ref{solvable-saturated-cor}, we obtain a good understanding of the structure below the categoricity cardinal of an AEC with amalgamation. For the convenience of the reader, we have added a few statements that we have already proven. We quote freely and refer the reader to the sources for more motivation on the results. We will use the following notation from \cite[Chapter 14]{baldwinbook09}:

\begin{notation}\label{hanf-def}
  $H_1 := \beth_{\left(2^{\LS (\K)}\right)^+}$.
\end{notation}

\begin{cor}\label{big-cor}
  Let $\lambda > \LS (\K)$. Assume that $\K_{<\lambda}$ has amalgamation and no maximal models. If $\K$ is semisolvable in $\lambda$, then:

  \begin{enumerate}
  \item\label{big-cor-1} For any $\mu \in [\LS (\K), \lambda)$, $\K$ is $\mu$-superstable and has $\mu$-symmetry.
  \item\label{big-cor-2} For any $\mu \in [\LS (\K), \lambda)$, any $M_0, M_1, M_2 \in \K_{\mu}$, if $M_1$ and $M_2$ are limit over $M_0$, then $M_1 \cong_{M_0} M_2$.
  \item\label{big-cor-3} For any $\mu \in [\LS (\K), \lambda)$, the union of any increasing chain of $\mu$-saturated models is $\mu$-saturated.
  \item\label{big-cor-4} If $\K$ is solvable in $\lambda$, then there exists an EM blueprint $\Psi \in \Upsilon_{\LS (\K)}[\K]$ such that $\EM_{\tau (\K)} (I, \Psi)$ is saturated for any linear order $I$ of cardinality in $(\LS (\K), \lambda]$.
  \item\label{big-cor-5} If $\K$ is solvable in $\lambda$ and either $\cf{\lambda} > \LS (\K)$ or $\lambda \ge H_1$, then there exists $\chi < \min (\lambda, H_1)$ such that:
    \begin{enumerate}
      \item\label{big-cor-5a} $\K$ is $(\chi, <\lambda)$-weakly tame.
      \item\label{big-cor-5b} For any $\mu \in (\chi, \lambda)$, there is a type-full good $\mu$-frame with underlying AEC the saturated models in $\K_{\mu}$.
    \end{enumerate}
  \end{enumerate}
\end{cor}
\begin{proof}
  Item (\ref{big-cor-1}) is Fact \ref{shvi} and Corollary \ref{sym-cor}. As for (\ref{big-cor-2}), let $\mu \in [\LS (\K), \lambda)$. By the previous part, $\K$ is $\mu$-superstable and has $\mu$-symmetry. By the main result of \cite{vandieren-symmetry-apal}, this implies uniqueness of limit models as stated here.

    Items (\ref{big-cor-3}) and (\ref{big-cor-4}) are part of the proof of Corollary \ref{solv-downward}. As for (\ref{big-cor-5a}), we use the relevant facts (due to Shelah) which assumes that the model in the categoricity (or just solvability) cardinal is saturated. They appear in \cite[Claim IV.7.2]{shelahaecbook} and \cite[Main claim II.2.3]{sh394} (depending on whether $\cf{\lambda} > \LS (\K)$ or $\lambda \ge H_1$), see also \cite[Theorem 8.4]{downward-categ-tame-apal}. Now (\ref{big-cor-5b}) follows from (\ref{big-cor-5a}) by \cite[Theorem 6.4]{vv-symmetry-transfer-afml}.
\end{proof}
\begin{remark}
  Corollary \ref{big-cor}.(\ref{big-cor-2}) proves \cite[Theorem 3.3.7]{shvi635} with the additional assumption that the class has amalgamation and improves on \cite[Corollary 7.3]{vv-symmetry-transfer-afml} which assumed that the categoricity cardinal $\lambda$ was ``big-enough''. See Section \ref{uq-limit-subsec} for more on the uniqueness of limit models.
\end{remark}

\subsection{Some categoricity transfers}\label{categ-subsec}

We mention improvements on several existing categoricity transfers. The partial downward transfer below improves on \cite[Corollary 7.7]{vv-symmetry-transfer-afml} and \cite[Corollary 9.7]{downward-categ-tame-apal}. The essence of the proof is a powerful omitting type theorem of Shelah \cite[Lemma II.1.6]{sh394}. Indeed the result is already implicit in \cite{sh394} when the categoricity cardinal $\lambda$ is regular (see also \cite[Theorem 14.9]{baldwinbook09}).

\begin{cor}
  Let $\K$ be an AEC with arbitrarily large models and let $\lambda > \LS (\K)$. Assume that $\K_{<\lambda}$ has amalgamation and no maximal models. If $\K$ is categorical in $\lambda$, then there exists $\chi < H_1$ such that $\K$ is categorical in any cardinal of the form $\beth_\delta$, where $\delta$ is divisible by $\chi$ and $\beth_\delta < \lambda$.
\end{cor}
\begin{proof}
  By (the proof of) \cite[Corollary 9.7]{downward-categ-tame-apal}, using that the model of categoricity $\lambda$ is saturated (Corollary \ref{solvable-saturated-cor}).
\end{proof}

We can also improve the thresholds of Shelah's proof of the eventual categoricity conjecture in AECs with amalgamation \cite[Theorem IV.7.12]{shelahaecbook} assuming the weak generalized continuum hypothesis. Shelah showed (assuming an unpublished claim) that in an AEC with amalgamation, categoricity in \emph{some} $\lambda \ge \hanfe{\aleph_{\LS (\K)^+}}$ implies categoricity in \emph{all} $\lambda' \ge \hanfe{\aleph_{\LS (\K)^+}}$.

Shelah's proof was revisited and expanded on in \cite[Section 11]{downward-categ-tame-apal}, from which we quote. Here, we improve the main lemma to:

\begin{lem}\label{main-lem}
  Assume an unpublished claim of Shelah \cite[Claim 11.2]{downward-categ-tame-apal}. Assume that $\K$ has arbitrarily large models. Let $\lambda \ge \mu > \LS (\K)$. Assume that $\K_{<\lambda}$ has amalgamation. If:

  \begin{enumerate}
  \item $\K$ is categorical in $\lambda$.
  \item $\mu$ is a limit cardinal with $\cf{\mu} > \LS (\K)$.
  \item For unboundedly many $\chi < \mu$, $2^{\chi^{+n}} < 2^{\chi^{+(n + 1)}}$ for all $n < \omega$.
  \end{enumerate}

  Then there exists $\mu_\ast < \mu$ such that $\K$ is categorical in any $\lambda' \ge \min(\lambda, \hanfe{\mu_\ast})$.
\end{lem}
\begin{proof}
  As in the proof of \cite[Fact 11.10]{downward-categ-tame-apal}, using that we know that the model of categoricity $\lambda$ is saturated (it is shown there that we can assume without loss of generality that $\K_{<\lambda}$ has no maximal models).
\end{proof}

We deduce that one can start with $\lambda \ge \aleph_{\LS (\K)^+}$ instead of $\lambda \ge \hanfe{\aleph_{\LS (\K)^+}}$.

\begin{cor}\label{shelah-improvement}
  Assume an unpublished claim of Shelah \cite[Claim 11.2]{downward-categ-tame-apal} and $2^{\mu} < 2^{\mu^+}$ for all cardinals $\mu$. Assume that $\K$ has arbitrarily large models. Let $\lambda \ge \aleph_{\LS (\K)^+}$ be such that $\K_{<\lambda}$ has amalgamation. If $\K$ is categorical in $\lambda$, then $\K$ is categorical in any $\lambda' \ge \min (\lambda, \hanfe{<\aleph_{\LS (\K)^+}})$.
\end{cor}
\begin{proof}
  Set $\mu := \aleph_{\LS (\K)^+}$ in Lemma \ref{main-lem}.
\end{proof}

We showed in \cite[Corollary 11.9]{downward-categ-tame-apal} that if $\K$ is $\LS (\K)$-tame and has amalgamation, then categoricity in some $\lambda \ge H_1$ implies categoricity in all $\lambda' \ge H_1$ (still assuming weak GCH and Shelah's unpublished claim). In \cite[Corollary B.7]{downward-categ-tame-apal}, we showed that it was consistent (using additional cardinal arithmetic assumptions) that one could replace tameness by just weak tameness. Here we prove it unconditionally.

\begin{cor}
  Assume an unpublished claim of Shelah \cite[Claim 11.2]{downward-categ-tame-apal} and there exists $\mu < \aleph_{\LS (\K)^+}$ such that $2^{\mu^{+n}} < 2^{\mu^{+(n + 1)}}$ for all $n < \omega$. Assume that $\K$ is $(\LS (\K), <H_1)$-weakly tame and has arbitrarily large models. Let $\lambda \ge \aleph_{\LS (\K)^+}$ be such that $\K_{<\lambda}$ has amalgamation. If $\K$ is categorical in $\lambda$, then there exists $\chi < H_1$ such that $\K$ is categorical in any $\lambda' \ge \min (\lambda, \chi)$.
\end{cor}
\begin{proof}
  Proceed as in the proof of \cite[Corollary B.7]{downward-categ-tame-apal} (as before, we can assume without loss of generality that $\K_{<\lambda}$ has no maximal models, hence the model of cardinality $\lambda$ is saturated). We use the better transfer we have just proven (Corollary \ref{shelah-improvement}).
\end{proof}

We also obtain more information on the author's categoricity transfer in universal classes \cite{ap-universal-v11-toappear, categ-universal-2-selecta}. There it was shown \cite[Theorem 7.3]{categ-universal-2-selecta} that if a universal class $\K$ is categorical in \emph{some} $\lambda \ge \beth_{H_1}$, then it is categorical in \emph{all} $\lambda' \ge \beth_{H_1}$. The reason that the threshold is $\beth_{H_1}$ rather than $H_1$ is that the proof works inside an auxiliary AEC $\K^\ast$ whose Löwenheim-Skolem-Tarski number is around $H_1$. A closer look at the proof reveals that $\LS (\K^\ast)$ is related to the length of a failure of the order property, so we can use Theorem \ref{op-thm} to improve the bound on $\LS (\K^\ast)$. We are unable to do so unconditionally so will assume that the class has no maximal models:

\begin{lem}\label{univ-nmm-cor}
  Let $\K$ be a universal class. Set $\mu := 2^{2^{\LS (\K)}}$, $\chi_1 := \beth_{\mu^{++}}$, $\chi_2 := \beth_{\left(2^{\mu^+}\right)^+}$. Let $\lambda > \chi_1$. If $\K$ is categorical in $\lambda$ and $\K_{<\lambda}$ has no maximal models\footnote{It suffices to assume that for every $M \in \K_{\LS (\K)^+}$ there exists $N \in \K_\lambda$ with $M \lea N$.} , then there exists $\chi < \chi_2$ such that $\K$ is categorical (and has amalgamation) in all $\lambda' \ge \min (\lambda, \chi)$.
\end{lem}
\begin{proof}[Proof sketch]
  First observe that $\chi_1 \ge H_1$, so $\K$ has arbitrarily large models. Second, by Theorem \ref{op-thm} and the no maximal models hypothesis, for any $\kappa < \aleph_0$, $\K$ does not have the $(\kappa, \LS (\K))$-order property of length $\LS (\K)^+$.

  We now follow the proof of \cite[Theorems 7.2, 7.3]{categ-universal-2-selecta}. We define an auxiliary class $\K^\ast$ which will have Löwenheim-Skolem-Tarski number $\left(2^{2^{\chi_0}}\right)^+$, where $\chi_0 \ge \LS (\K)$ is least such that $\K$ does not have a syntactic version of the order property of length $\chi_0^+$. It is straightforward to see that if $\K$ has the order property (in the sense there) of length $\chi_0^+$, then for some $\kappa < \aleph_0$ $\K$ has the $(\kappa, 0)$-order property (in the sense of Definition \ref{op-def}) of length $\chi_0^+$. This means that $\chi_0 = \LS (\K)$, and hence $\LS (\K^\ast) = \mu^+$.

  Now $\K^\ast$ may not satisfy the smoothness axiom of AECs and to ensure this the proof of \cite[Theorem 7.2]{categ-universal-2-selecta} uses categoricity in a $\lambda$ with $\left(\beth_{\mu^{++}}\right)^+ < \lambda$. However if $\lambda = \left(\beth_{\mu^{++}}\right)^+$, then $\lambda$ is regular so by \cite[Theorem IV.1.11]{sh300-orig} (building many models in the categoricity cardinal from failure of smoothness) we also get that $\K^\ast$ is an AEC. Therefore $\K^\ast$ is an AEC whenever $\lambda > \beth_{\mu^{++}} = \chi_1$. and we can then continue exactly as in the proof of \cite[Theorem 7.3]{categ-universal-2-selecta}.
\end{proof}

A more quotable version of Lemma \ref{univ-nmm-cor} is below. Compared to \cite[Theorem 7.3]{categ-universal-2-selecta}, $\beth_{\beth_{\left(2^{\LS (\K)}\right)^+}}$ is replaced by the much lower $\beth_{\beth_{5} (\LS (\K))}$.

\begin{cor}
  Let $\K$ be a universal class with no maximal models. If $\K$ is categorical in \emph{some} $\lambda \ge \beth_{\beth_5 (\LS (\K))}$, then $\K$ is categorical in \emph{all} $\lambda' \ge \beth_{\beth_5 (\LS (\K))}$.
\end{cor}
\begin{proof}
  In the statement of Lemma \ref{univ-nmm-cor}, $\chi_1 < \chi_2 \le \beth_{\beth_5 (\LS (\K))}$.
\end{proof}
\begin{remark}
  One can ask what happens if instead of no maximal models, one makes the stronger assumption of amalgamation below the categoricity cardinal. Then we obtain the best possible result as in \cite[Corollary 10.11]{downward-categ-tame-apal} (this is proven using amalgamation also above the categoricity cardinal, but we can use \cite[Theorem 4.16]{ap-universal-v11-toappear} to get away with just amalgamation below).
\end{remark}

\subsection{More on the uniqueness of limit models}\label{uq-limit-subsec}

The main result of \cite{shvi635} claims that assuming no maximal models and GCH (with instances of $\Diamond$), limit models of cardinality $\mu$ are unique for any $\mu$ below the categoricity cardinal $\lambda$. In VanDieren's Ph.D.\ thesis \cite{vandierenthesis}, two additional hypotheses that it seemed the proof needed were identified (see \cite{vandierennomax} for background and terminology):

\begin{enumerate}
\item\label{problem-1} The union of any increasing chain of limit models in $\K_\mu$ of length less than $\mu^+$ is a limit model.
\item\label{problem-2} Reduced towers in $\K_\mu$ are continuous.
\end{enumerate}

In \cite[Theorem III.10.3]{vandierennomax}, VanDieren claimed to prove (\ref{problem-2}) assuming (\ref{problem-1}). VanDieren later found a gap \cite{nomaxerrata} and fixed the gap assuming in addition that $\lambda = \mu^+$. Here we prove the original statement of \cite[Theorem III.10.3]{vandierennomax} (still using (\ref{problem-1})).

The key is the next result which generalizes Corollary \ref{sym-cor}. Note that we do not assume (\ref{problem-1}). Note also that, below, we use only the model-theoretic consequences (in the context described above \cite{shvi635}) of GCH and appropriate instances of $\Diamond$. Finally, note that we have replaced the assumption of categoricity in $\lambda$ mentioned above by the weaker assumption of semisolvability in $\lambda$ (see Remark \ref{solv-rmk}).

\begin{cor}\label{sym-thm}
  Let $\lambda > \LS (\K)$. Assume that $\K_{<\lambda}$ has no maximal models. Assume that $\K$ semisolvable in $\lambda$ and fix $\mu \in [\LS (\K), \lambda)$. If in $\K_{\mu}$ amalgamation bases are dense, universal extensions exist, and limit models are amalgamation bases, then $\K$ has $\mu$-symmetry (Definition \ref{sym-def} has to be slightly adapted by asking for $N$ to be an amalgamation base).
\end{cor}
\begin{proof}[Proof sketch]
  We follow the proof of Corollary \ref{sym-cor}. Fact \ref{shvi} was originally proven in the context here, and the proof of Fact \ref{op-sym-fact} shows that failure of $\mu$-symmetry implies the $(2, \mu)$-order property of length $\mu^+$ (in that case it is not clear that the construction can be continued all the way to $\lambda$). Theorem \ref{op-thm} (with $\kappa$ there standing for $2$ here) and the no maximal models hypothesis shows that $\K$ cannot have the $(2, \mu)$-order property of length $\mu^+$, so symmetry in $\mu$ must hold.
\end{proof}

We obtain the desired proof of \cite[Theorems II.9.1, III.10.3]{vandierennomax}. This also generalizes Corollary \ref{big-cor}.(\ref{big-cor-2}).

\begin{cor}\label{uq-limit-shvi}
  Let $\lambda > \LS (\K)$. Assume that $\K_{<\lambda}$ has no maximal models. Assume that $\K$ is semisolvable in $\lambda$ and fix $\mu \in [\LS (\K), \lambda)$. If in $\K_{\mu}$ amalgamation bases are dense, universal extensions exist, limit models are amalgamation bases, and (\ref{problem-1}) above holds\footnote{or just: the union of any increasing chain of limit models in $\K_{\mu}$ of length less than $\mu^+$ is an \emph{amalgamation base}.}, then reduced towers in $\K_\mu$ are continuous. In particular, whenever $M_0, M_1, M_2 \in \K_{\mu}$ are such that $M_1$ and $M_2$ are limit models over $M_0$, we have that $M_1 \cong_{M_0} M_2$.
\end{cor}
\begin{proof}
  By Corollary \ref{sym-thm}, $\K$ has symmetry in $\mu$. By the proof of \cite[Theorem 3]{vandieren-symmetry-apal}, reduced towers in $\K_{\mu}$ are continuous. As pointed out in \cite{nomaxerrata}, the proof of \cite[Theorem II.9.1]{vandierennomax} now goes through to prove the uniqueness of limit models.
\end{proof}

\bibliographystyle{amsalpha}
\bibliography{categ-saturated}

\end{document}